\documentclass{article}
\usepackage{graphicx,amssymb,amsmath}
\usepackage[curve]{xypic}

\title{Quantifiers for quantum logic}
\author{Chris Heunen}

\newcommand{\after}{\circ}
\newcommand{\cat}[1]{\ensuremath{\mathbf{#1}}}
\newcommand{\Cat}[1]{\ensuremath{\mathbf{#1}}}
\newcommand{\id}[1][]{\ensuremath{\mathrm{id}_{#1}}}
\newcommand{\op}{\ensuremath{^{\mathrm{op}}}}
\newcommand{\field}[1]{\ensuremath{\mathbb{#1}}}

\newcommand{\tensor}{\ensuremath{\otimes}}
\newcommand{\ie}{\textit{i.e.}~}
\newcommand{\eg}{\textit{e.g.}~}
\newcommand{\coker}{\ensuremath{\mathrm{coker}}}
\newcommand{\Sub}{\ensuremath{\mathrm{Sub}}}

\newcommand{\Hilb}{\Cat{Hilb}}

\newcommand{\ClSub}{\ensuremath{\mathrm{ClSub}}}
\makeatother
\newcommand{\pullback}[1][dr]{\save*!/#1-1.2pc/#1:(-1,1)@^{|-}\restore}
 \newdir{ >}{{}*!/-7.5pt/@{>}}
 \newdir{|>}{!/4.5pt/@{|}*:(1,-.2)@^{>}*:(1,+.2)@_{>}}
 \newdir{ |>}{{}*!/-3pt/@{|}*!/-7.5pt/:(1,-.2)@^{>}*!/-7.5pt/:(1,+.2)@_{>}}
\makeatletter
\newtheorem{theorem}{Theorem}
\newtheorem{lemma}{Lemma}
\newtheorem{proposition}{Proposition}
\newtheorem{corollary}{Corollary}
\newtheorem{definition}{Definition}

\newenvironment{proof}[1][Proof]%
{ \begin{trivlist}%
  \item[\hskip \labelsep {\bfseries #1}]%
}%
{ \end{trivlist}%
}
\newcommand{\qed}{\nobreak\hfill$\Box$}
\newcommand{\qedhere}{\tag*{$\Box$}}

\begin{document}
\maketitle

\begin{abstract}
  We consider categorical logic on the category of Hilbert spaces.
  More generally, in fact, any pre-Hilbert category suffices.  
  We characterise closed subobjects, and prove that they
  form orthomodular lattices. This shows that quantum logic is just an
  incarnation of categorical logic, enabling us to establish an   
  existential quantifier for quantum logic, and conclude that there
  cannot be a universal quantifier.
\end{abstract}

\section{Introduction}

Quantum logic is the study of closed subspaces of a Hilbert
space~\cite{birkhoffvonneumann:quantumlogic}. Intriguingly, this
`logic' is not distributive, but only satisfies the weaker axiom of
orthomodularity. One of the shortcomings that has kept it from
wide adoption is the lack of quantifiers. In fact, it has been called
a `non-logic'~\cite{abramsky:temperleylieb}. 

On the other hand, categorical
logic~\cite{lambekscott:categoricallogic} can be seen as a unified 
framework for any kind of logic that deserves the name. 
It is concerned with interpreting (syntactical) logical formulae in
categories with enough structure to accommodate this.
An important part of it is the study of subobjects of a given object
in the category at hand. Perhaps its most gratifying feature is that
it gives a canonical prescription of what quantifiers should be.

The aim of this paper is to show that quantum logic is just an
incarnation of categorical logic in categories like that of Hilbert
spaces. In particular, we will establish an existential quantifier,
and conclude that there cannot be a universal quantifier.

Section~\ref{sec:hilbertcats} first abstracts the properties
of the category of Hilbert spaces that we need. This results in an
axiomatisation of \emph{(pre-)Hilbert categories} greatly resembling
that of monoidal Abelian categories. In fact, any (pre-)Hilbert category
embeds into the category of (pre-)Hilbert spaces
itself~\cite{heunen:hilbcatsembedding}. Next,
Section~\ref{sec:subobjects} starts the investigation of subobjects in
Hilbert categories.  It turns out that the natural objects of study
are not the subobjects, but the  \emph{closed subobjects} or
\emph{$\dag$-subobjects}. Section~\ref{sec:existentialquantor} then
derives a functor that behaves as an existential quantifier according to
categorical logic. Section~\ref{sec:orthogonality} studies the
emergent concept of orthogonality in Hilbert categories. First, it
proves that $\dag$-subobjects form orthomodular lattices. Second, it
exhibits a tight connection between adjoint morphisms in the base
category and adjoint functors between the lattices of subobjects, the
latter being important in connection to quantifiers. 

\subsubsection*{Related work}

The present article should not be confused with the `categorical quantum
logic' of~\cite{duncan:typesforquantumcomputing}. That work develops a
type theory. Of course this is related: ``every logic is a logic over
a type theory''~\cite{jacobs:fibrations}. This paper develops the
logic over `the type theory of Hilbert spaces'.

This paper also differs from \cite{harding:orthomodularity}, in that
the aim is explicitly a categorical logic. Another difference is that
that paper restricts to those projections that have an
orthocomplement, whereas we derive orthomodularity from prior
assumptions (namely $\dag$-kernels).

\section{Pre-Hilbert categories}
\label{sec:hilbertcats}

This section introduces the categories in which our study takes place,
somewhat concisely. For more information we refer
to~\cite{heunen:hilbcatsembedding}. 

A functor $\dag:\cat{H}\op \to \cat{H}$ with $X^\dag=X$ on objects and 
$f^{\dag\dag}=f$ on morphisms is called a \emph{$\dag$-functor}; the
pair $(\cat{H}, \dag)$ is then called a \emph{$\dag$-category}. Such
categories are automatically isomorphic to their opposite, and the
$\dag$-functor witnesses this self-duality.
We can consider coherence of the $\dag$-functor with all sorts of
structures. 
A morphism $m$ in such a category that satisfies $m^\dag m = \id$ is
called a \emph{$\dag$-mono} and denoted $\xymatrix@1{\ar@{ |>->}[r] &
}$. Likewise, $e$ is a \emph{$\dag$-epi}, denoted
$\xymatrix@1{\ar@{-|>}[r] & }$, when $ee^\dag=\id$. A morphism is
called a \emph{$\dag$-iso} when it is both $\dag$-epi and $\dag$-mono.
Similarly, a biproduct
on such a category is called a \emph{$\dag$-biproduct} when
$\pi^\dag = \kappa$, where $\pi$ is a projection and $\kappa$ an
injection. This is equivalent to demanding $(f \oplus g)^\dag = f^\dag
\oplus g^\dag$.
Also, an equaliser is called a \emph{$\dag$-equaliser} when
it is a $\dag$-mono, and a kernel is called a \emph{$\dag$-kernel}
when it is a $\dag$-mono. Finally, a $\dag$-category $\cat{H}$ is
called \emph{$\dag$-monoidal} when it is 
equipped with monoidal structure $(\tensor,C)$ that cooperates with
the $\dag$-functor, in the sense that $(f \tensor g)^\dag = f^\dag
\tensor g^\dag$, and the coherence isomorphisms are
$\dag$-isomorphisms. 

\begin{definition}
\label{def:hilblikecategory}
  A category is called a \emph{pre-Hilbert category} when
  \begin{itemize}
    \item it has a $\dag$-functor;
    \item it has finite $\dag$-biproducts;
    \item it has (finite) $\dag$-equalisers;
    \item every $\dag$-mono is a $\dag$-kernel; and
    \item it is symmetric $\dag$-monoidal.
  \end{itemize}
\end{definition}

Notice that a Hilbert category is self-dual (by the $\dag$-functor),
and therefore that it automatically has all finite colimits, too.

The category $\Cat{preHilb}$ itself is a pre-Hilbert category whose
monoidal unit is a simple generator, and so are its full subcategories
$\Cat{Hilb}$, and $\Cat{fdHilb}$ of finite-dimensional Hilbert
spaces. Also, if $\cat{C}$ is a small category and $\cat{H}$ a
pre-Hilbert category, then $[\cat{C}, \cat{H}]$ is again a pre-Hilbert
category. Working in pre-Hilbert categories can be thought of as
`natural' or `baseless' (pre-)Hilbert space theory.

\section{Subobjects}
\label{sec:subobjects}

This section characterises closed subobjects categorically.
But let us start with some easy properties of $\dag$-mono's.

\begin{lemma} 
\label{lem:dagmonoproperties} 
  In any $\dag$-category: 
  \begin{enumerate}
  \renewcommand{\labelenumi}{\emph{(\alph{enumi})}}
    \item A $\dag$-mono which is epi is a $\dag$-iso.
    \item The composite of $\dag$-epi's is again a $\dag$-epi.
    \item If $\smash{\xymatrix@1{X \ar^-{f}[r] & Y \ar^-{g}[r] & Z}}$ are such
      that both $gf$ and $f$ are $\dag$-epi, so is $g$. 
    \item If $m$ and $n$ are $\dag$-monos, and $f$ is an iso with
      $nf=m$, then $f$ is a $\dag$-iso. 
  \end{enumerate}
\end{lemma}
\begin{proof}
  For (a), notice that $ff^\dag=\id$ implies $ff^\dag f=f$,
  from which $f^\dag f=\id$ follows from the assumption that $f$ is epi.
  For (b): $gf(gf)^\dag = g f f^\dag g^\dag = g^\dag g = \id$.
  And for (c): $gg^\dag = gff^\dag g = gf (gf)^\dag = \id$.
  Finally, consider (d). If $f$ is iso, in particular it is epi. If
  both $nf$ and $n$ are $\dag$-mono, then so is $f$, by (c). Hence by
  (a), $f$ is $\dag$-iso.
  \qed  
\end{proof}

\emph{From now on, we work in an arbitrary pre-Hilbert category $\cat{H}$.}

\begin{lemma}
  \label{lem:Abepi}
  A morphism $m$ is mono iff $\ker(m)=0$.
  Consequently, if $mf=0$ implies $f=0$ for all $f$, then $m$ is mono.
\end{lemma}
\begin{proof}
  Suppose $\ker(m)=0$. Let $u,v$ satisfy $mu=mv$. Put $q$ to be the
  $\dag$-coequaliser of 
  $u$ and $v$. Since $q$ is a $\dag$-epi, $q=\coker(w)$ for some $w$.
  As $mu=mv$, $m$ factors through $q$ as $m=nq$.
  Then $mw=nqw=n0=0$, so $w$ factors through $\ker(m)$ as $w=\ker(m)
  \after p$ for some $p$. But since $\ker(m)=0$, $w=0$. So $q$ is a
  $\dag$-iso, and in particular mono. Hence, from $qu=qv$ follows
  $u=v$. Thus $m$ is mono.
  
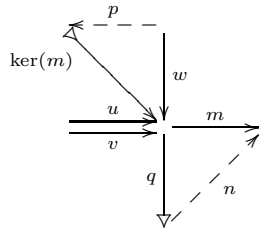
\[\xymatrix@C+2ex@R+2ex{
    \ar@{ |>->}_<<<{\ker(m)}[dr] & \ar@{-->}_-{p}[l] \ar^-{w}[d] \\
    \ar@<.5ex>^-{u}[r] \ar@<-.5ex>_-{v}[r] 
    & \ar^-{m}[r] \ar@{-|>}_-{q}[d] & \\
    & \ar@{-->}_-{n}[ur] 
  }\]
  Conversely, if $m$ is mono, it follows from $m \after \ker(m) = 0 =
  m \after 0$ that $\ker(m)=0$.
   
  If $f=0$ whenever $mf=0$, then $\ker(m) = 0$, so that $m$ is mono. 
  \qed
\end{proof}

\subsection{Factorisation}

This subsection proves that any morphism $f:X \to Y$ in a pre-Hilbert
category can be factorised as an epi $e:X \to I$ followed by a
$\dag$-mono $m:I \to Y$. (In $\Cat{Hilb}$, this is very easily proved
concretely: $e$ is simply the restriction of $f$ to $I$, the closure
of its range, and $m$ is the isometric inclusion of $I$ into $Y$.)
Recall that since a pre-Hilbert category has $\dag$-kernels, it
automatically also has $\dag$-cokernels by $\coker(f) =
\ker(f^\dag)^\dag$. 

\begin{lemma}
  \label{lem:factorisation}
  Any pre-Hilbert category has a factorisation system consisting of
  mono's and $\dag$-epi's. The factorisation is unique up to a unique
  $\dag$-iso. Consequently, every $\dag$-epi is a $\dag$-cokernel of
  its $\dag$-kernel. 
\end{lemma}
\begin{proof} 
  Let a morphism $f$ be given.
  Put $k=\ker(f)$ and $e=\coker(k)$.
  Since $fk=0$ (as $k=\ker(f)$), $f$ factors through $e$($=\coker(k)$)
  as $f=me$.
  
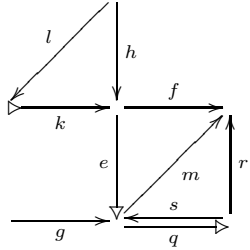
\[\xymatrix@C+3ex@R+3ex{
    & \ar^-{h}[d] \ar_-{l}[dl] \\
    \ar@{ |>->}_-{k}[r] & \ar@{-|>}_-{e}[d] \ar^-{f}[r] & \\
    \ar_-{g}[r] & \ar_-{m}[ur] \ar@{-|>}@<-.5ex>_-{q}[r] 
    & \ar_-{r}[u] \ar@<-.25ex>_-{s}[l]
  }\]
  We have to show that $m$ is mono.
  Let $g$ be such that $mg=0$. By Lemma~\ref{lem:Abepi} 
  it suffices to show that $g=0$.
  Since $mg=0$, $m$ factors through $q=\coker(g)$ as $m=rq$.
  Now $qe$ is a $\dag$-epi, being the composite of two $\dag$-epi's.
  So $qe=\coker(h)$ for some $h$.
  Since $fh=rqeh=r0=0$, $h$ factors through $k$($=\ker(f)$) as
  $h=kl$. 
  Finally $eh=ekl=0l=0$, so $e$ factors through $qe=\coker(h)$ as
  $q=sqe$. 
  But since $e$ is a ($\dag$-)epi, this means $sq=\id$, whence $q$ is
  mono.
  It follows from $qg=0$ that $g=0$, and the factorisation is
  established. 
  
  Since $\dag$-epi's are regular epi's, and hence strong epi's, 
  functoriality of the factorisation follows
  from~\cite[4.4.5]{borceux:1}. By
  Lemma~\ref{lem:dagmonoproperties}d, the factorisation is unique up
  to a $\dag$-iso.

  Finally, suppose that $f$ is a $\dag$-epi. Then both the above 
  $f = m \after e$ and $f = f \after \id$ are mono-$\dag$-epi
  factorisations of $f$. Hence $f=e$ up to the unique mediating
  $\dag$-iso $m$, showing that $f=\coker(\ker(f))$.
  \qed
\end{proof}

We just showed that any pre-Hilbert category has a factorisation
system consisting of mono's and $\dag$-epi's. Equivalently, it has a
factorisation system of epi's and $\dag$-mono's. Indeed, if we can
factor $f^\dag$ as an $\dag$-epi followed by a mono, then taking the
daggers of those, we find that $f^{\dag\dag}=f$ factors as an epi
followed by a $\dag$-mono. The combination of both factorisations
yields that every morphism can be written as a $\dag$-epi, followed by
a monic epimorphism, followed by a $\dag$-mono; this can be thought of
generalising \emph{polar decomposition}. 

\subsection{Closed subobjects, pullbacks}

A \emph{subobject} of an object $X$ in a $\dag$-category is an
equivalence class of mono's $m:M \rightarrowtail X$, where $m$ is
equivalent to $n:N \rightarrowtail X$ if there is an isomorphism $f:M
\to N$ satisfying $nf=m$. The class of subobjects of $X$ is denoted
$\Sub(X)$. It is partially ordered by $M \leq N$ iff there is a 
morphism $f:M \to N$ with $nf=m$. It also has a largest element,
represented by $\id[X]:X \to X$. Because a pre-Hilbert category has
pullbacks, $\Sub(X)$ is in fact a meet-semilattice\footnote{We 
disregard size issues here. A $\dag$-category is called
\emph{$\dag$-well-powered} if $\ClSub(X)$ is a set for all objects $X$
in it. Since $\ClSub(X)$ for $X \in \Cat{Hilb}$ is the set of closed
subspaces of $X$, $\Cat{Hilb}$ is $\dag$-well-powered.}, the meet of
$M$ and $N$ being represented by the pullback of $m$ and $n$. 
Moreover, for each $f:X \to Y$, pullback along $f$ induces a
meet-preserving map $f^{-1} : \Sub(Y) \to \Sub(X)$.
Thus we have a functor $\Sub : \cat{H}\op \to \Cat{MeetSLat}$, the
\emph{inverse image functor}.  

A \emph{$\dag$-subobject} is a subobject that can be represented by a
$\dag$-mono. We write $\ClSub(X)$ for the class of $\dag$-subobjects of
$X$. It inherits the partial ordering of
$\Sub(X)$. It can be characterised precisely when a subobject $m$ is a
$\dag$-subobject, namely when there is an isomorphism $\varphi$ such
that $m^\dag m = \varphi^\dag \varphi$~\cite[5.6]{selinger:daggeridempotents}. 

\begin{lemma}
\label{lem:dsubstableunderpullbacks}
  $\dag$-subobjects are stable under pullbacks.
\end{lemma}
\begin{proof}
  Let $\xymatrix@1{n:N \ar@{ |>->}[r] & Y}$ and $f:X \to Y$.
  Consider their pullback $(P,p,q)$. Factorise $p$ as
  $p=\ker(\coker(p)) \after e$ with $e$ epi. 
  \[\xymatrix{
    & P \ar@{->>}_-{e}[dl] \ar@{ >->}^-{p}[dd] \ar^-{q}[rr] \pullback
    && N \ar@{ |>->}^-{n=\ker(\coker(n))}[dd] \\
    I \ar@{ |>->}_-{\ker(\coker(p))}[dr] \ar@{-->}_-{g}[urrr] \\
    & X \ar_-{f}[rr] && Y
  }\]
  Now 
  \begin{align*}
        \coker(n) \after f \after \ker(\coker(p)) \after e 
    & = \coker(n) \after f \after p \\
    & = \coker(n) \after n \after q \\
    & = 0 \after q = 0.
  \end{align*}
  Since $e$ is epi, $\coker(n) \after f \after \ker(\coker(p)) = 0$.
  So there is a $g:I \to N$ satisfying $n \after g = f \after
  \ker(\coker(p))$. Since $P$ is a pullback, there is a $h:I \to P$
  such that $\ker(\coker(p))=p \after h$. 
  Now $\ker(\coker(p)) = p \after h = \ker(\coker(p)) \after e \after
  h$ and $\ker(\coker(p))$ is mono, so $e \after h = \id[P]$.
  Also $p = \ker(\coker(p)) \after e = p \after h \after e$ and $p$ is
  mono, so $h \after e = \id[I]$.
  Hence $e$ is iso, and $p=f^{-1}(n) \in \ClSub(X)$.
  \qed
\end{proof}

Hence every morphism $f:X \to Y$ induces a meet-preserving map $f^{-1}
: \ClSub(Y) \to \ClSub(X)$. Thus we have a functor
\[
  \ClSub : \cat{H}\op \to \Cat{MeetSLat},
\]
that we also call the \emph{inverse image functor} with abuse of
terminology. 

Recall that a \emph{closure operation}~\cite{birkhoff:latticetheory} 
consists in giving for every $m \in \Sub(X)$ a $\overline{m} \in
\Sub(X)$, satisfying (i) $m \leq \overline{m}$, (ii) $m \leq n
\Rightarrow \overline{m} \leq \overline{n}$, and (iii)
$\overline{\overline{m}} = \overline{m}$.

\begin{lemma}
\label{lem:closure}
  $m \mapsto \ker(\coker(m))$ is a closure operation. 
\end{lemma}
\begin{proof}
  For (i): $\coker(m) \after m=0$, so $m \leq \ker(\coker(m))$.
  For (ii): if $m \leq n$, then $\coker(m) \after \ker(\coker(m)) =
  0$,
  \[\xymatrix@R-2ex@C+10ex{
    M \ar^-{m}[dr] \ar[dd] & & \\
    & X \ar^-{\coker(n)}[ur] \ar^-{\coker(m)}[dr] \\
    N \ar^-{n}[ur] & & \ar[uu]  \\
    & \ar_-{0}[ur] \ar^-{\ker(\coker(n))}[uu]
  }\]
  so $\ker(\coker(m)) \leq \ker(\coker(n))$.
  For (iii): since $\ker(\coker(m)) \in \ClSub(X)$, we have
  $\ker(\coker(\ker(\coker(m)))) = \ker(\coker(m))$ by
  Lemma~\ref{lem:factorisation}. 
  \qed
\end{proof}

\begin{lemma}
\label{lem:subindsub}
  There is a reflection
  $\xymatrix@1{ \Sub(X) \ar@<1ex>^-{\ker(\coker(-))}[rr] \ar@{}|\perp[rr]
  && \;\ClSub(X) \ar@<1ex>@{_{(}->}[ll] }$.
\end{lemma}
\begin{proof}
  We have to prove that $\ker(\coker(m)) \leq n$ iff $m \leq n$ for a
  mono $m$ and a $\dag$-mono $n$.
  By (i) of Lemma~\ref{lem:closure} we have $m \leq
  \ker(\coker(m))$, proving one direction. The converse direction is
  just (ii) of Lemma~\ref{lem:closure}. 
  \qed
\end{proof}

The previous lemma could be interpreted as a moral justification for
studying the (replete) semilattice of closed subobjects instead of
that of subobjects.

\subsection{Projections}

Instead of closed subobjects, it turns out we can also consider projections. 
A \emph{projection} on $X$ is a morphism $p:X \to X$ satisfying 
$p\after p = p = p^\dag$. We define $\mathrm{Proj}(X)$ as the set of all 
projections on $X$. It is partially ordered by defining $p \sqsubseteq q$ 
iff $p \after q = p$.

\begin{proposition}
  There is an order isomorphism $\ClSub(X) \cong \mathrm{Proj}(X)$.
\end{proposition}
\begin{proof}
  Any closed subobject $m$ yields a projection $mm^\dag$.
  Conversely, any projection $p$ gives a closed subobject $\Im(p)$.
  
  Let us verify that these maps are each others inverses. Starting with a 
  closed subobject represented by $m$, we end up with $\Im(mm^\dag)$. 
  Since $m$ is $\dag$-mono and $m^\dag$ is $\dag$-epi, this is already a
  factorisation in the sense of Lemma~\ref{lem:factorisation}, and hence 
  $\Im(mm^\dag)=m$ as closed subobjects.
  Conversely, a projection $p$ maps to $i i^\dag$,
  where $p$ factors as $p=i e$ for an epi $e:X \to I$ and $\dag$-mono 
  $i=\Im(p)$. By functoriality of the factorisation it follows from $pp=p$ 
  that $pi=i$. Now 
  \[
    i=pi=p^\dag i = (ie)^\dag i = e^\dag i^\dag i = e^\dag,
  \]
  so indeed $ii^\dag = ie = p$.

  Finally let us consider the order. If $m \leq n$ as subobjects, say 
  $m=n\varphi$ for a $\dag$-mono $\varphi$, then $mm^\dag nn^\dag = 
  n\varphi \varphi^\dag n^\dag nn^\dag = nn^\dag nn^\dag = nn^\dag$, so 
  indeed $mm^\dag \sqsubseteq nn^\dag$.
  Conversely, if $p \sqsubseteq q$, then $pq=p$, whence $\Im(pq)=\Im(p)$, so
  that indeed $\Im(p) \leq \Im(q)$ by functoriality of the factorisation.
  \qed
\end{proof}

Consequently, every result we derive about the partial order of closed 
subobjects holds for the projections and vice versa.

\section{Existential quantifier}
\label{sec:existentialquantor}

This section establishes an existential quantifier, \ie a left adjoint
to the inverse image functor that satisfies the Beck-Chevalley condition.

\begin{proposition}
  $\ClSub(X)$ is a lattice.
\end{proposition}
\begin{proof}
  Since we already know that $\ClSub(X)$ is a meet-semilattice, it
  suffices to show that it has joins and a least element.
  Joins follow from \eg \cite[4.2.6]{borceux:1}. Explicitly,
  $\xymatrix@1@C+2ex{M \vee N\; \ar@{ |>->}^-{m \vee n}[r] &  X}$ is given by
  $\Im(s)$, where $s=[m,n]:M \oplus N \to X$.
  The closed smallest subobject, the bottom element of $\ClSub(X)$, is
  given by $\xymatrix@1{0\; \ar@{ |>->}^-{0}[r] & X}$.
  \qed
\end{proof}

The $\dag$-mono $m:M \rightarrowtail Y$ arising in the factorisation
of a morphism $f:X \to Y$ of $\cat{H}$ is called the \emph{(direct)
image} of $f$, denoted $\mathrm{Im}(f)$. Notice that $\mathrm{Im}(f)$
defines a unique $\dag$-subobject, 
although the representing $\dag$-mono is only unique up to a
$\dag$-iso. This $\dag$-subobject is denoted $\exists_f$. More
generally, for $\xymatrix@1{n:N\; \ar@{ |>->}[r] & X}$ in $\ClSub(X)$, we
define  
\[
  \exists_f(N) = \mathrm{Im}(f n),
\]
which gives a well-defined map $\exists_f : \ClSub(X) \to \ClSub(Y)$ for
any morphism $f:X \to Y$ of $\cat{H}$.

\begin{theorem}
  \label{thm:exists}
  Let $f:X \to Y$ be a morphism of $\cat{H}$. 
  The map $\exists_f:\ClSub(X) \to \ClSub(Y)$ is monotone and
  left-adjoint to $f^{-1} : \ClSub(Y) \to \ClSub(X)$. If $g:Y \to Z$ is
  another morphism then $\exists_g \after \exists_f = \exists_{g
  \after f} : \ClSub(X) \to \ClSub(Z)$. Also $\exists_{\id} = \id$.
\end{theorem}
\begin{proof}
  We follow the proof of \cite[Lemma~2.5]{butz:regular}. 
  For monotonicity of $\exists_f$ let $M \leq N$ in $\ClSub(X)$. First
  factorise $n$ and then $M \to \exists_fN$ to get the following diagram.
  \[\xymatrix{
    X \ar^-{f}[r] & Y \\
    N \ar@{->>}[r] \ar@{ |>->}_-{n}[u] & \exists_f N \ar@{ |>->}[u] \\
    M \ar@{->>}[r] \ar@{ |>->}[u] \ar@(l,l)@{ |>->}^-{m}[uu] &
    I \ar@{ |>->}[u]
  }\]
  Now $\xymatrix@1{M \ar@{->>}[r] & I\; \ar@{ |>->}[r] & Y}$ is an
  epi-$\dag$-mono factorisation of $fm$, so $I$ represents
  $\exists_fM$, and $\exists_fM \leq \exists_fN$.

  To show the adjunction, let $M \in \ClSub(X)$ and $N \in \ClSub(Y)$,
  and consider the solid arrows in the following diagram.
  \[\xymatrix{
    & X \ar^-{f}[rr] && Y \\
    & f^{-1}N \pullback[ur] \ar@{ |>->}[u] \ar|(.775)\hole[rr] 
    && N \ar@{ |>->}_-{n}[u] \\
    M \ar@{->>}[rr] \ar@{ |>->}^-{m}[uur] \ar@{-->}[ur]
    && \exists_f M \ar@{-->}[ur] \ar@{ |>->}[uur]
  }\]
  If $\exists_f M \leq N$ then the right dashed map $\exists_f M \to N$
  exists and the outer square commutes. Hence, since $f^{-1}N$ is a
  pullback, the left dashed map $M \to f^{-1}N$ exists, and $M \leq
  f^{-1}N$. Conversely, if $M \leq f^{-1}N$, factorise the map $M \to
  N$ to get the image of $M$ under $f$. In particular, this image then
  factors through $N$, whence $\exists_f M \leq N$.

  Finally, the identity $\exists_g \after \exists_f = \exists_{g
  \after f}$ just states how left adjoints compose.
  \qed
\end{proof}

\subsection{The Beck-Chevalley condition}

Recall the \emph{Beck-Chevalley condition}: if the left square
below is a pullback, then the right one must commute.
\begin{equation}
\label{eq:beckchevalley}\tag{BC}
  \raise5ex\hbox{\xymatrix{
    P \pullback \ar^-{q}[r] \ar_-{p}[d] & Y \ar^-{g}[d] \\
    X \ar_-{f}[r] & Z    
  }} \qquad \Rightarrow \qquad
  \raise5ex\hbox{\xymatrix{
    \ClSub(P) \ar_-{\exists_p}[d] & \ClSub(Y) \ar^-{\exists_g}[d]
    \ar_-{q^{-1}}[l] \\
    \ClSub(X) & \ClSub(Z) \ar^-{f^{-1}}[l] 
  }}
\end{equation}
It ensures that the semantics of the existential quantifier is sound
with respect to substitution. To show that our $\exists_f$
satisfies~\eqref{eq:beckchevalley}, we will assume that the monoidal
unit $C$ of our pre-Hilbert category $\cat{H}$ is a simple generator.
Recall that an object $C$ is called a \emph{generator} when $fx=gx$
for all $x:C \to X$ implies $f=g:X \to Y$. It is called \emph{simple}
when $\Sub(C) = \{0,C\}$. In this case,
\cite[Theorem~4.6]{heunen:hilbcatsembedding} shows that $\cat{H}$ is
enriched over Abelian groups, so that we can talk of adding and
subtracting morphisms.

\begin{lemma}
  In a pre-Hilbert category whose monoidal unit is a simple generator,
  epi's are stable under pullback.
\end{lemma}
\begin{proof}
  The proof of~\cite[Proposition~1.7.6]{borceux:1} works verbatim.
  \qed
\end{proof}

The previous lemma entails that $\cat{H}$ is a \emph{regular
category}, and hence that all results of~\cite{butz:regular} apply.
Thus, in such a category $\cat{H}$ one can soundly interpret regular
logic, in particular the existential quantifier.

\begin{theorem}
\label{thm:beckchevalley}
  In a pre-Hilbert category whose monoidal unit is a simple generator,
  \eqref{eq:beckchevalley} holds.  
\end{theorem}
\begin{proof}
  The proof of~\cite[Lemma~2.9]{butz:regular} works verbatim. 
  \qed
\end{proof}

Also the \emph{Frobenius identity} holds. Let $f:X \to Y$ be a
morphism of $\Cat{Hilb}$. Let $M \in \ClSub(X)$ and $N \in
\ClSub(Y)$. Then $\exists_f(M \wedge f^{-1}N) = \exists_fM \wedge N$
as $\dag$-subobjects of $Y$. For a proof, we refer
to~\cite[Lemma~2.6]{butz:regular}. 


\section{Orthogonality}
\label{sec:orthogonality}

We will now recover the orthogonal subspace construction from the
$\dag$-functor in any pre-Hilbert category. The idea is to mimick the
fact that $\ker(f)^\perp = \Im(f^\dag)$ in $\Cat{Hilb}$. 

\begin{proposition}
  \label{lem:perpfromdagger}
  There is an involutive functor $(-)^\perp : \ClSub(X)\op \to \ClSub(X)$
  determined by $m^\perp = \ker(m^\dag)$ for $m \in \ClSub(X)$.
\end{proposition}
\begin{proof}
  To show that the above definition extends functorially, let $m,n \in
  \ClSub(X)$ be such that $m \leq n$. Say that $m$ factors through $n$
  by $m=ni$ for $i:M \to N$. Then
  \[
      m^\dag \after \ker(n^\dag) 
    = i^\dag \after n^\dag \after \ker(n^\dag)
    = i^\dag \after 0
    = 0.
  \]
  Hence $\ker(n^\dag)$ factors through $\ker(m^\dag)$, that is,
  $n^\perp \leq m^\perp$.

  We finish the proof by showing that $\perp$ is involutive:
  \[
      m^{\perp\perp} 
    = (\ker(m^\dag))^\perp
    = \ker(\ker(m^\dag)^\dag)
    = \ker(\coker(m))
    = m.
  \]
  Here, the last equation follows from Lemma~\ref{lem:factorisation}.
  \qed
\end{proof}

The functor $(-)^\perp$ cooperates with $\wedge$ and $\vee$ as expected.

\begin{lemma}
  \label{lem:orthocomplement}
  $\ClSub(X)$ is an orthocomplemented lattice, that is,
  $m \wedge m^\perp = 0$ and $m \vee m^\perp=1$ for all $m \in
  \ClSub(X)$.
  (A forteriori, the cotuple $[m,m^\perp]$ is a $\dag$-iso.)
\end{lemma}
\begin{proof}
  Recall that $m \wedge m^\perp$ is defined as the $\dag$-pullback
  \[\xymatrix@C+3ex@R+1ex{
      M \wedge M^\perp \pullback \ar@{ |>-->}^-{p}[r] \ar@{ |>-->}_-{q}[d] 
    & M^\perp \ar@{ |>->}^-{\ker(m^\dag)}[d] \\
      M \ar@{ |>->}_-{m}[r] & X
  }\]
  Because $m$ is a $\dag$-mono, we have $q = m^\dag \after m \after q
  = m^\dag \after \ker(m^\dag) \after p = 0 \after p = 0$. 
  Hence $m \wedge m^\perp = m \after q = m \after 0 = 0$.

  To prove the second claim, let $f$ satisfy $f \after
  [m,m^\perp]=0$. 
  Then $f \after m=0$, so $f$ factors through $\coker(m)$ as $f=g
  \after \coker(m)$. 
  Also $f \after m^\perp=0$, so 
  \[
      g 
    = g \after \ker(m^\dag)^\dag \after \ker(m^\dag) 
    = f \after \ker(m^\dag) 
    = 0,
  \]
  whence $f=0$.
  So, by Lemma~\ref{lem:Abepi}, $[m,m^\perp]$ is epi. 
  Hence $[m,m^\perp]$ factors as $\id \after [m,m^\perp]$, but also as
  $(m \vee m^\perp) \after p$. So $m \vee m^\perp$ must be a
  $\dag$-iso. That is, $m \vee m^\perp=1$.
  
  Let us prove that $[m,m^\perp]$ is also a $\dag$-mono, and hence
  even a $\dag$-iso:
  \begin{align*}
        [m,m^\perp]^\dag \after [m,m^\perp]
    & = \langle m^\dag, \ker(m^\dag)^\dag \rangle \after [m, \ker(m^\dag)] \\
    & = \left( \begin{array}{cc}
          m^\dag \after m & m^\dag \after \ker(m^\dag) \\
          \ker(m^\dag)^\dag \after m & \ker(m^\dag)^\dag \after \ker(m^\dag) 
        \end{array} \right) \\
    & = \id[M \oplus M^\perp].
  \qedhere
  \end{align*}
\end{proof}

However, $(-)^\perp$ has poor `substitution properties', as it does
not commute with pullbacks. For a counterexample in $\Hilb$, let
$X=\field{C}^2$, $Y=\field{C}$, $f=\pi:X\to Y : (x,y) \mapsto x$ and
$m=0:0 \to Y$. Then $f^{-1}(m^\perp) = \field{C}^2$, but
$(f^{-1}(m))^\perp = \{(x,0) \mid x \in \field{C}\}$.

In spite of this, a special case of ``$(-)^\perp$ is stable under
pullbacks'' still holds: we now recover orthomodularity of $\ClSub(X)$
using the previous lemma. 

\begin{theorem}
  $\ClSub(X)$ is an orthomodular lattice, that is, $m \vee (m^\perp
  \wedge n) = n$ whenever $m \leq n \in \ClSub(X)$.
\end{theorem}
\begin{proof}
  Say that $m$ factors through $n$ as $m=nd$. Let $(p,q)$ be the
  pullback of $\ker(m^\dag)$ along $n$. Then $n \after [q,d] = [nq,nd]
  = [p \after \ker(m^\dag), m]$. Hence, if we can show that $[q,d]$ is
  epi, this would be the factorisation of $[p \after \ker(m^\dag),m]$,
  and $m \vee (m^\perp \wedge n)=n$. 

  First we show that $M \oplus (M^\perp \wedge N)$ is a
  pullback of $n$ and $[m,m^\perp]$. Let $f$ and $g$ satisfy
   $n \after g = [m, \ker(m^\dag)] \after f$.
  Then there are unique $h_1,h_2$ making the following diagram
  commute. 
  \[\xymatrix@R+2ex@C+2ex{
    &&& P \ar_-{g}[d] \ar@{-->}^-{h_2}[dr] \ar@{-->}_-{h_1}[dl] 
      \ar `[rr] 
          `[ddrr]^{\pi_2 \after f}
           [ddr] 
      \ar `^d[ll]
          `^r[ddll]_{\pi_1 \after f}
           [ddl]
    &&
    \\
    &&M^\perp \wedge N \pullback \ar@{ |>->}^-{q}[r] \ar@{ |>->}_-{p}[d]
    & N \ar@{ |>->}_-{n}[d]
    & M \ar@{ |>->}_-{d}[l] \ar@{=}^-{\id}[d] \pullback[ld] && \\
    &&M^\perp \ar@{ |>->}_-{\ker(m^\dag)}[r] & X & M \ar@{ |>->}^-{m}[l] &&
  }\]
  This makes the following into a pullback diagram.
  \[\xymatrix{
    P \ar@(r,u)^-{g}[drr] \ar@{-->}^{\langle h_1,h_2 \rangle}[dr] 
      \ar@(d,l)_-{f}[ddr]\\
    & M \oplus (M^\perp \wedge N) \pullback \ar^-{[d,q]}[r] \ar_-{\id
    \oplus p}[d] & N \ar@{ |>->}^-{n}[d] \\
    & M \oplus M^\perp \ar_-{[m,\ker(m^\dag)]}[r] & X
  }\]
  Since isomorphisms are stable under pullback, $[d,q]$ is iso. In
  particular, it is epi, and the theorem is established.
  \qed
\end{proof}

\begin{corollary}
  There cannot be right adjoints $f^{-1} \dashv \forall_f$ for all
  morphisms $f$ of $\cat{H}$, that satisfy the Beck-Chevalley condition.
\end{corollary}
\begin{proof}
  If there would be, then $\wedge$ would have a right adjoint in every
  $\ClSub(X)$~\cite[3.4.16]{awodeybauer:introduction}. That is, there
  would be an implication. But the prime example $\Hilb$ shows that
  $\ClSub(X)$ is in general not a Heyting algebra.
  \qed
\end{proof}

\begin{lemma}
  The functor $\perp : \ClSub(X)\op \to \ClSub(X)$ is an equivalence of
  categories. In particular, it is both left and right adjoint to
  its opposite $\perp\op : \ClSub(X) \to \ClSub(X)\op$.
\end{lemma}
\begin{proof}
  This means precisely that $m^\perp \leq n$ iff $n^\perp \leq m$,
  which holds since $\perp$ is involutive.
  \qed
\end{proof}

The following theorem, inspired by~\cite{palmquist:adjoints}, provides
a connection between adjoint morphisms in a pre-Hilbert category and
adjoint functors between lattices of $\dag$-subobjects. 
It explicates the relationship between $\exists_f$ and $\exists_{f^\dag}$.

\begin{theorem}
\label{thm:existsandadjoints}
  For a morphism $f:X \to Y$, define 
  \[
    f^\perp = \perp_Y \after \exists_f\op : \ClSub(X)\op \to \ClSub(Y).
  \]
  Then
  \[
    (f^\perp)\op \dashv (f^\dag)^\perp.
  \]  
\end{theorem}
\begin{proof}
  In general, for $g:Y \to X$, the adjunction $(f^\perp)\op \dashv
  g^\perp$ means that for $M \in \ClSub(X)$ and $N \in \ClSub(Y)$,
  \[ 
    \ClSub(Y)\op(\perp_Y\op \after \exists_f(M),N) 
    \cong
    \ClSub(X)(M,\perp_X \after \exists_{f^\dag}\op(N)).
  \]
  That is, $n \leq \ker(\Im(fm)^\dag)$ iff $m \leq
  \ker(\Im(gn)^\dag)$. That means that in 
  \begin{equation}
  \label{eq:adjointexistsfunctors}\raise12ex\hbox{
  \xymatrix@R-3.5ex@C+3ex{
    L \ar^-{0}[rr] \ar@{ |>->}^-{\ker(l^\dag)}[ddr] 
    && J \ar@{ >->}[ddr]^-{j^\dag} \\ \\
    M \ar@{-->}^-{q}[uu] \ar@{ |>->}^-{m}[r] 
    & X \ar@{-|>}^-{l^\dag}[uur] \ar^-{g^\dag}[r]
    & Y \ar@{-|>}^-{n^\dag}[r] & N \\
    M \ar@{ |>->}_-{m}[r] \ar@{->>}_-{i}[ddr] 
    & X \ar_-{f}[r] & Y \ar@{-|>}_-{n^\dag}[r]
    \ar@{-|>}_-{\coker(k)}[ddr] & N \\ \\
    & I \ar@{ |>->}_-{k}[uur] \ar_-{0}[rr]
    && K \ar@{-->}_-{p}[uu]
  }}\end{equation}
  we must show that there is a $p$ making the lower diagram commute iff
  there is a $q$ making the upper one commute, for the special case
  $g=f^\dag$. So, let such a $q$ be given. Then 
  \[
      n^\dag \after k \after i
    = n^\dag \after f \after m
    = j \after l^\dag \after m
    = j \after l^\dag \after \ker(l^\dag) \after q
    = j \after 0 \after q
    = 0
    = 0 \after i,
  \]
  and since $i$ is epi, $n^\dag k = 0$. Hence $n^\dag$ factors through
  $\coker(k)$ via some $p$. Conversely, given $p$, we have
  \[
      j^\dag \after l^\dag \after m
    = n^\dag \after f \after m
    = n^\dag \after k \after i
    = p \after \coker(k) \after k \after i
    = p \after 0 \after i
    = 0
    = j^\dag \after 0,
  \]
  so since $j^\dag$ is mono, $l^\dag m = 0$. Hence $m$ factors through
  $\ker(l^\dag)$ via some $q$.
  \qed
\end{proof}

In a diagram, the adjunction of the previous theorem is the following.
\[\xymatrix{
  \ClSub(X) \ar@{}|{\rotatebox{45}{$\vdash$}}[dr]
  \ar^-{\exists_f}[r] & \ClSub(Y) \ar^-{\perp_Y\op}[d] \\
  \ClSub(X)\op \ar^-{\perp_X}[u] & \ClSub(Y)\op
  \ar^-{\exists_{f^\dag}\op}[l] 
}\]
A converse to this theorem needs some preparation, and the assumption
that the monoidal unit is a simple generator.

\begin{lemma}
\label{lem:functionals}
  Let $C$ be a simple object in a pre-Hilbert category.
  If $f,g:X \to C$ satisfy $\ker(f) \leq \ker(g)$,
  then $g=sf$ for some $s:C \to C$. Unless $f=0$, this $s$ is unique.
\end{lemma}
\begin{proof}
  Consider $\exists_fX \in \ClSub(C)$. Either $\exists_fX=0$, or
  $\exists_fX$ is an iso and hence a $\dag$-iso since it is a $\dag$-mono. 

  If $\exists_fX=0$, then $f=0$. So $\ker(f)$ is a
  $\dag$-iso, and since $\ker(f)\leq \ker(g)$, also $\ker(g)$ is
  $\dag$-iso, whence $g=0$. Thus $g = 0 f$.

  If $\exists_fX$ is a $\dag$-iso, in particular it is epi, and so is $f$.
  It can be factorised as a $\dag$-epi $f'$ followed by a mono $s_f$.
  \[\xymatrix@R-4ex@C+2ex{
    & I \ar@{ >->}^-{s_f}[dr] \\
    X \ar@{-|>}^-{f'}[ur] \ar@{->>}_-{f}[rr] & & C
  }\]
  Now either $s_f=0$ or $s_f$ is iso. If $s_f=0$ then $\exists_fX=0$ and
  hence $f=0$, so that we are done by $g = 0 f$. Hence we may assume
  $s_f$ iso.

  Since $\ker(f') \leq \ker(f) \leq \ker(g)$
  we are thus left with the following situation.
  \[\xymatrix@C+3ex@R-3ex{
    L \ar@{ |>->}^-{\ker(g)}[dr] & & C \\
    & X \ar@{-|>}^-{f'}[ur] \ar_-{g}[dr] \\
    K \ar@{ |>->}^-{p}[uu] \ar@{ |>->}_-{\ker(f')}[ur] & & C
  }\]
  Now $f' = \coker(\ker(f'))$, and 
  \[
       g \after \ker(f') 
     = g \after \ker(g) \after p 
     = 0 \after p 
     = 0.
  \]
  Hence there is a unique $s'$ such that $g = s' \after f'$.
  Finally, putting $s = s' s_f^{-1}$ satisfies
  $g = s' f' = s' s_f^{-1} f = s f$.
  \qed
\end{proof}

In a monoidal category, morphisms $s:C \to C$ play
the role of scalars, and multiplication with them is natural. 
As mentioned before, if $C$ is a simple generator, then the scalars
comprise an involutive
field~\cite[Theorem~4.6]{heunen:hilbcatsembedding}. The following
lemma summarises some well-known (and easily proved) results. 

\begin{lemma}
\label{lem:scalarmultiplication}
  Let $\cat{H}$ be a monoidal category.
  Then $\cat{H}(C,C)$ is an involutive semiring that acts on $\cat{H}$
  by \emph{scalar multiplication} as follows: for $s:C \to C$ and $f:X
  \to Y$, $s \bullet f$ is defined by
  \[\xymatrix@R-2ex{
    X \ar^-{s \bullet f}[r] \ar_-{\cong}[d] & Y \\
    C \tensor X \ar_-{s \tensor f}[r] & C \tensor Y \ar_-{\cong}[u]
  }\]
  Moreover, scalar multiplication is natural, that is, 
  $(s \bullet g) \after f = g \after (s \bullet f)$.
  Finally, $s \bullet f = s \after f$ for $s:C \to C$ and $f:X \to C$.
  \qed
\end{lemma}

Now we can state and prove a converse to
Theorem~\ref{thm:existsandadjoints}. 

\begin{theorem}
\label{thm:existsandadjointsconverse}
  In a pre-Hilbert category whose monoidal unit is a simple generator,
  if $(f^\perp)\op \dashv g^\perp$, then $g = s \bullet f^\dag$ for a
  scalar $s$. Unless $f=0$, this $s$ is unique.
\end{theorem}
\begin{proof}
  The adjunction of the hypothesis
  means that there is a $q$ making the upper diagram
  in~\eqref{eq:adjointexistsfunctors} commute iff there is a $p$
  making the lower one commute.
  So, if $n^\dag f m=0$, then $n^\dag k i=0$, and because $i$ is epi
  hence $n^\dag k=0$. So $p$ exists, whence $q$ exists, so that
  $n^\dag g^\dag m = j^\dag 0 q = 0$. Taking $m=\ker(n^\dag f)$ thus
  gives that $\ker(n^\dag f) \leq \ker(n^\dag g^\dag)$ for all $n$. 
  Applying Lemma~\ref{lem:functionals} yields that for all
  $n:C \to Y$, there exists $s_n :C \to C$ such that $n^\dag g^\dag =
  s_n n^\dag f$. Using Lemma~\ref{lem:scalarmultiplication} and
  dualising, this becomes: for all $n:C \to Y$, there is $s_n : 
  C \to C$ with $g n = (s_n^\dag \bullet f^\dag) n$. 
  We will show that all $s_n$ are in fact equal to each other (or zero).
  If all $y:C \to Y$ would have $y=0$, then $Y \cong 0$, in which case
  $g=0 \bullet f^\dag$. Otherwise, pick an $y:C \to Y$ with $y \neq
  0$. There is an $s:C \to C$ with $gy = (s^\dag \bullet f^\dag)y$. 
  Put $n' = y y^\dag n : C \to Y$ and $n''=\ker(y^\dag) \after
  \ker(y^\dag)^\dag \after n:C \to Y$. Then
  \begin{align*}
        n' + n''
    & = [\id,\id] \after ((y \after y^\dag \after n) \oplus
        (\ker(y^\dag) \after \ker(y^\dag)^\dag \after n)) \after 
        \langle \id, \id \rangle \\
    & = [ y, y^\perp] \after [y, y^\perp]^\dag \after n \\
    & = n.
  \end{align*}
  Moreover, 
  \[ 
    (s_{n'}^\dag \bullet f^\dag)n' = gn' = gyy^\dag n = (s^\dag \bullet
    f^\dag) yy^\dag n = (s^\dag \bullet f^\dag) n',
  \]
  so $s_{n'} = s$. Finally
  \begin{align*}
         (s_{n'}^\dag \bullet f^\dag) n' + (s_{n''}^\dag \bullet
         f^\dag) n''
     & = gn' + gn'' \\
     & = gn \\
     & = (s_n^\dag \bullet f^\dag) n \\
     & = (s_n^\dag \bullet f^\dag) n' + (s_n^\dag \bullet f^\dag) n''.
  \end{align*}
  Hence $s_n=s_{n'}=s$ for all $n:C \to Y$, and we have $g n = (s^\dag
  \bullet f^\dag) n$. But since $C$ is a generator, $g = s^\dag
  \bullet f^\dag$. Reviewing our choice of $s$ in the above proof, we
  see that it is unique unless $f=0$.
  \qed
\end{proof}

As a consequence, we find that, modulo scalars, the passage from
morphisms $f$ to functors $\perp \after \exists_f\op$ is one-to-one.

\appendix
\section{Fibred account}

We can summarise our results in terms of fibred category
theory~\cite{jacobs:fibrations}. There are fibrations $\Sub(\cat{H})
\to \cat{H}$ and $\ClSub(\cat{H}) \to \cat{H}$. The latter is in fact a
fibration of meet-semilattices by
Lemma~\ref{lem:dsubstableunderpullbacks}. The reflection of
Lemma~\ref{lem:subindsub} is a fibred reflection. Our functor
$\exists$ of Theorem~\ref{thm:exists} is a fibred coproduct, and hence
truely provides a existential quantifier. 

The assignments
$\cat{H} \to \Sub(\cat{H})$, $X \mapsto \id[X]$ assemble into a fibred
terminal object $1:\cat{H} \to \Sub(\cat{H})$, also for $\cat{H} \to
\ClSub(\cat{H})$. 
The fibrations $\Sub(\cat{H}) \to \cat{H}$ and $\ClSub(\cat{H}) \to
\cat{H}$ admit comprehension. This means that $1:\cat{H} \to
\Sub(\cat{H})$ has a right adjoint, usually denoted by
$\{-\}:\Sub(\cat{H}) \to \cat{H}$. Indeed, if we take $\{m : M
\rightarrowtail X\} = M$, then $\Sub(\cat{H})(\id[X], m) \cong
\cat{H}(X, \{m\})$.

In fact, the fibration $\ClSub(\cat{H}) \to \cat{H}$ is a bifibration
by Theorem~\ref{thm:exists} and \cite[9.1.2]{jacobs:fibrations} --
notice that the Beck-Chevalley condition is not needed for this. 
Thus, $\ClSub(\cat{H})\op \to \cat{H}\op$, $\xymatrix@1{(m:M\;\ar@{
|>->}[r] & X)} \mapsto X$ is also a fibration. The following
proposition shows that orthogonality can be extended to a functor
between fibrations, but it is not a fibred functor, basically because
it does not commute with pullback. 

\begin{proposition}
  $(-)^\perp$ extends to a functor $\ClSub(\cat{H})\op \to
  \ClSub(\cat{H})$ satisfying
  \begin{equation}\label{eq:perpfibfunctor}\raise5ex\hbox{\xymatrix{
    \ClSub(\cat{H})\op \ar[d] \ar^-{(-)^\perp}[r] & \ClSub(\cat{H}) \ar[d] \\
    \cat{H}\op \ar_-{(-)^\dag}[r] & \cat{H}
  }}\end{equation}
  However, it is not a fibred functor.  
\end{proposition}
\begin{proof}
  We can understand $(-)^\perp$ as a functor $\ClSub(\cat{H})\op \to
  \ClSub(\cat{H})$ by extending its action on morphisms as follows. Let
  $(f,g)$ be a morphism $m \to n$, that is, let $f:X \to Y$ and $g:M
  \to N$ satisfy $fm=ng$. We are to define a morphism $(f,g)^\perp :
  n^\perp \to m^\perp$, that is, a pair $f^\perp:Y \to X$ and
  $g^\perp:N^\perp \to M^\perp$ satisfying $f^\perp \after n^\perp =
  m^\perp \after g^\perp$. Put $f^\perp = f^\dag$. Then  
  \[
      m^\dag \after f^\dag \after n^\perp
    = g^\dag \after n^\dag \after \ker(n^\dag)
    = g^\dag \after 0 = 0,
  \]
  so there is a $g^\perp$ such that $f^\dag \after n^\perp =
  \ker(m^\dag) \after g^\perp = m^\perp \after g^\perp$. It must be
  \[
      g^\perp
    = (m^\perp)^\dag \after m^\perp \after g^\perp
    = \ker(m^\dag)^\dag \after f^\dag \after n^\perp
    = \coker(m) \after f^\dag \after \ker(n^\dag).
  \]
  This explicitly defines the functor $(-)^\perp : \ClSub(\cat{H})\op \to
  \ClSub(\cat{H})$. It makes the square~\eqref{eq:perpfibfunctor}
  commute.
  Now, a morphism $(f,g):m \to n$ of $\ClSub(\cat{H})$ is Cartesian (over
  $f$) iff $f=ngm^\dag=nn^\dag f mm^\dag$. Consequently, the morphism
  $(f,g)^\perp : n^\perp \to m^\perp$ in $\ClSub(\cat{H})\op$ is Cartesian iff 
  \[
      f^\dag 
    = \ker(m^\dag) \after \ker(m^\dag)^\dag \after f^\dag \after \ker(n^\dag)
      \after \ker(n^\dag)^\dag.
  \]
  Thus, $(-)^\perp$ is a fibred functor iff $f^\dag = \ker(m^\dag)
  \ker(m^\dag)^\dag f^\dag \ker(n^\dag) \ker(n^\dag)^\dag$ whenever
  $f=nn^\dag f mm^\dag$ for any morphism $f$ and $\dag$-mono's $m$ and
  $n$.

  Finally, we come to our counterexample. Take $m=\kappa_M$ for $M \neq
  0$, $f=mm^\dag$ and $n=\id[M\oplus M]$. Then $f=mm^\dag =mm^\dag
  mm^\dag = nn^\dag f mm^\dag$. But $\ker(m^\dag) = \kappa'$ so
  $\ker(m^\dag)^\dag = \pi'$ and $\ker(n^\dag)=0$, so
  \begin{align*}
      \ker(m^\dag) \after \ker(m^\dag)^\dag \after f^\dag \after \ker(n^\dag)
      \after \ker(n^\dag)^\dag.
    = \kappa' \after \pi' \after f^\dag \after 0
    = 0
    \neq f^\dag.
  \end{align*}
  Hence $(-)^\perp$ is not a fibred functor.
  \qed
\end{proof}

\end{document}